\documentclass[preprint,12pt]{amsart}
\usepackage{amsmath, amssymb, amsthm}

\theoremstyle{plain}
\newtheorem{theorem}{Theorem}
\newtheorem{corollary}[theorem]{Corollary}
\newtheorem{lemma}[theorem]{Lemma}
\newtheorem{proposition}[theorem]{Proposition}
\newtheorem{problem}[theorem]{Problem}
\newtheorem{conjecture}[theorem]{Conjecture}

\theoremstyle{definition}
\newtheorem{definition}[theorem]{Definition}
\newtheorem{example}[theorem]{Example}
\newtheorem{remark}[theorem]{Remark}
\newtheorem{question}[theorem]{Question}

\DeclareMathOperator{\card}{card}

\DeclareMathOperator{\id}{id}
\newcommand{\acr}{\newline\indent}

\begin{document}

\title{On monoids of metric preserving functions}

\author{Viktoriia Bilet}
\address{\textbf{Viktoriia Bilet}\acr
Department of Theory of Functions \acr
Institute of Applied Mathematics and Mechanics of NASU\acr
Dobrovolskogo str. 1, Slovyansk 84100, Ukraine}
\email{viktoriiabilet@gmail.com}

\author{Oleksiy Dovgoshey}
\address{\textbf{Oleksiy Dovgoshey}\acr
Department of Theory of Functions \acr
Institute of Applied Mathematics and Mechanics of NASU \acr
Dobrovolskogo str. 1, Slovyansk 84100, Ukraine \acr
and\acr
University of Turku\acr
FI-200014 Turun yliopisto, Finland}

\email{oleksiy.dovgoshey@gmail.com; oleksiy.dovgoshey@utu.fi}


\begin{abstract}
Let $\mathbf{X}$ be a class of metric spaces and let $\mathbf{P}_{\mathbf{X}}$ be the set of all $f:[0, \infty)\to [0, \infty)$ preserving $\mathbf{X},$ $(Y, f\circ\rho)\in\mathbf{X}$ whenever $(Y, \rho)\in\mathbf{X}.$ For arbitrary subset $\mathbf{A}$ of the set of all metric preserving functions we show that the equality $\mathbf{P}_{\mathbf{X}}=\mathbf{A}$ has a solution iff $\mathbf{A}$ is a monoid with respect to the operation of function composition. In particular, for the set $\mathbf{SI}$ of all amenable subadditive increasing functions there is a class $\mathbf{X}$ of metric spaces such that $\mathbf{P}_{\mathbf{X}}=\mathbf{SI}$ holds, which gives a positive answer to the question of paper \cite{Dov24}.
\end{abstract}

\maketitle

\section{Introduction}

The following is a particular case of the concept introduced by Jacek Jachymski and Filip Turo\-bo\'{s} in \cite{JTRRACEFNSAMR2020}.

\begin{definition}\label{def1}
Let $\mathbf{A}$ be a class of metric spaces. Let us denote by $\mathbf{P}_{\mathbf{A}}$ the set of all functions $f:[0, \infty)\to [0, \infty)$ such that the implication
\begin{equation*}\label{eqdef1}
\left((X, d)\in\mathbf{A}\right) \Rightarrow \left((X, f\circ d)\in\mathbf{A}\right)
\end{equation*}
is valid for every metric space $(X, d).$
\end{definition}

For mappings $F: X\to Y$ and $\Phi: Y\to Z$ we use the symbol $F\circ \Phi$ to denote the mapping
$$X \xrightarrow{F} Y \xrightarrow{\Phi} Z.$$


We also use the following notation:

$\mathbf{F},$ set of functions $f:[0,\infty)\to [0, \infty);$

$\mathbf{F}_{0},$ set of functions $f\in\mathbf{F}$ with $f(0)=0;$

$\mathbf{Am},$ set of amenable $f\in\mathbf{F};$

$\mathbf{SI},$ set of subadditive increasing $f\in\mathbf{Am};$

$\mathbf{M},$ class of metric spaces;

$\mathbf{U},$ class of ultrametric spaces;

$\mathbf{Dis},$ class of discrete metric spaces;

$\mathbf{M}_{2},$ class of two-points metric spaces;

$\mathbf{M}_{1},$ class of one-point metric spaces.








The main purpose of this paper is to give a solution of the following problems.

\begin{problem}\label{probl1}
Let $\mathbf{A}\subseteq\mathbf{P}_{\mathbf{M}}.$ Find conditions under which the equation
\begin{equation}\label{e1.prob1}
\mathbf{P}_{\mathbf{X}}=\mathbf{A}
\end{equation}
has a solution $\mathbf{X}\subseteq\mathbf{M}.$
\end{problem}



\begin{problem}\label{probl3}
Let $\mathbf{A}\subseteq\mathbf{P}_{\mathbf{U}}.$ Find conditions under which equation \eqref{e1.prob1} has a solution $\mathbf{X}\subseteq\mathbf{U}.$
\end{problem}

In addition, we find all solutions to equation \eqref{e1.prob1} for $\mathbf{A}$ equal to $\mathbf{F},$ $\mathbf{F}_{0},$ or $\mathbf{Am}$ and answer the following question.

\begin{question}\label{quest_1}
Is there a subclass $\mathbf{X}$ of the class $\mathbf{M}$ such that $$\mathbf{P}_{\mathbf{X}}=\mathbf{SI}?$$
\end{question}

This question was asked in \cite{Dov24} in a different but equivalent form and it was the original motivation for our research.

The paper is organized as follows. The next section contains some necessary definitions and facts from the theories of metric spaces and metric preserving functions.

In Section~3 we recall some definitions from the semigroup theory and describe solutions to equation \eqref{e1.prob1}, for the cases when $\mathbf{A}$ is $\mathbf{F},$ $\mathbf{F}_{0}$ or $\mathbf{Am}.$ 
In addition, we show that $\mathbf{P}_{\mathbf{X}}$ is always a submonoid of $(\mathbf{F}, \circ).$ See Theorems~\ref{l1}, \ref{Th(1)}, \ref{Th(2)} and Proposition~\ref{propvsp}, respectively.


Solutions to Problems~\ref{probl1} and \ref{probl3} are given, respectively, in Theorems~\ref{mainth} and~\ref{mainth_4} of Section~4. Theorem~\ref{mainth_2} gives a positive answer to Question~\ref{quest_1}.

\section{Preliminaries on metrics and metric preserving functions}


Let $X$ be nonempty set. A function $d: X\times X\to [0, \infty)$ is said to be a \emph{metric} on the set $X$ if for all $x, y, z \in X$ we have:
\begin{itemize}
\item[$(i)$] $d(x, y) \geqslant 0$ with equality if and only if $x = y$, the \emph{positivity property};

\item[$(ii)$] $d(x, y) = d(y, x)$, the \emph{symmetry property};

\item[$(iii)$] $d(x, y) \leqslant d(x, z) + d(z, y)$, the \emph{triangle inequality}.
\end{itemize}



A metric space $(X,d)$ is \emph{ultrametric} if the \emph{strong triangle inequality}
$$
d(x,y) \leqslant \max \{ d(x,z), d(z,y) \}
$$
holds for all $x,y,z \in X$.

\begin{example}\label{example_19}
Let us denote by $\mathbb{R}_{0}^{+}$ the set $(0, \infty)$. Then the mapping $d^+ \colon \mathbb{R}_{0}^{+} \times \mathbb{R}_{0}^{+} \to [0, \infty),$
$$
d^+(p,q) := \left\{
\begin{array}{ll}
0 & \quad \hbox{if}\quad p = q, \\
\max \{p,q\} & \quad \hbox{otherwise}.
\end{array}
\right.
$$
is the ultrametric on $\mathbb{R}_{0}^{+}$ introduced by C.~Delhomm\'{e}, C.~Laflamme, M.~Pouzet, and N.~Sauer in \cite{DLPS2008TaiA}.
\end{example}







\begin{definition}\label{defdiscr}
Let $(X, d)$ be a metric space. The metric $d$ is \emph{discrete} if there is $k\in (0, \infty)$ such that \begin{equation*}\label{dis} d(x, y)=k \end{equation*} for all distinct $x,y\in X.$
\end{definition}

In what follows we will say that a metric space $(X, d)$ is discrete if $d$ is a discrete metric on $X.$ We will denote by $\mathbf{Dis}$ the class of all discrete metric space. In addition, for given nonempty set $X_{1},$ we will denote by $\mathbf{Dis}_{X_1}$ the subclass of $\mathbf{Dis}$ consisting of all metric spaces $(X_1, d)$ with discrete $d.$

\begin{example}\label{ex(9)}
Let $\mathbf{M}_k,$ for $k=1,2,$ be the class of all metric spaces $(X, d)$ satisfying the equality $$\card(X)=k.$$ Then all metric spaces belonging to $\mathbf{M}_{1}\cup\mathbf{M}_{2}$ are discrete.
\end{example}

\begin{proposition}\label{Pr(10)}
The following statements are equivalent for each metric space $(X,d)\in\mathbf{M}.$
\begin{itemize}
\item[$(i)$] $(X,d)$ is discrete.
\item[$(ii)$] Every three-point subspace of $(X,d)$ is discrete.
\end{itemize}
\end{proposition}

\begin{proof}
The implication $(i)\Rightarrow(ii)$ is evidently valid.

Suppose that $(ii)$ holds but $(X,d)\not\in\mathbf{Dis}.$ Then there are some different points $i,j,k,l\in X$ such that
\begin{equation}\label{eq1_pr}
d(i,j)\ne d(k,l).
\end{equation}
Write $X_{1}:=\{i,j,k\}$ and $X_{2}:=\{j,k,l\}.$ Then the spaces $(X_{1}, d|_{X_{1}\times X_{1}})$ and $(X_{2}, d|_{X_{2}\times X_{2}})$ are discrete subspaces of $(X,d)$ by statement $(ii)$. Consequently we have
\begin{equation}\label{eq2_pr}
d(i,j)= d(j,k)
\end{equation} and
\begin{equation}\label{eq3_pr}
d(j,k)= d(k,l)
\end{equation}
by definition of the class $\mathbf{Dis}.$ Now \eqref{eq2_pr} and \eqref{eq3_pr} give us $$d(i,j)=d(k,l),$$ which contradicts \eqref{eq1_pr}.
\end{proof}

\begin{remark}
The standard definition of discrete metric can be formulated as: ``The metric on X is discrete if the distance from each point of
$X$ to every other point of $X$ is one.'' (See, for example, \cite[p.~14]{Sea2007}.)
\end{remark}

Let $\mathbf{F}$ be the set of all functions $f: [0, \infty)\to [0, \infty).$

\begin{definition} A function $f \in\mathbf{F}$ is \emph{metric preserving} (\emph{ultrametric preserving}) iff $f \in \mathbf{P}_{\mathbf{M}}$ ($f \in \mathbf{P}_{\mathbf{U}}$). \end{definition}

\begin{remark}\label{rem1}
The concept of metric preserving functions can be traced back to Wilson \cite{Wilson1935}. Similar problems were
considered by Blumenthal in \cite{Blumenthal1936}. The theory of metric preserving functions was developed by Bors\'{\i}k, Dobo\v{s}, Piotrowski, Vallin and other mathematicians \cite{BD1981MS, Borsik1988, Dobos1996, Dobos1994, Dobos1996a, Dobos1997, DM2013, V1997RAE, V1998AMUC, V1998IJMMS, Vallin2000, PT2014FPTA, V2002TMMP}.
See also lectures by Dobo\v{s} \cite{Dobos1998}, and the introductory paper by Corazza \cite{Corazza1999}. The study of ultrametric preserving functions begun by P.~Pongsriiam and I.~Termwuttipong in 2014~\cite{PTAbAppAn2014} and was continued in \cite{Dov2020MS, VD2021MS}.
\end{remark}

We will say that $f \in\mathbf{F}$ is \emph{amenable} iff
\begin{equation*}\label{eq2.1}
f^{-1}(0)=\{0\}
\end{equation*}
holds and will denote by $\mathbf{Am}$ the set of all amenable functions from $\mathbf{F}.$
Let us denote by $\mathbf{F}_{0}$ the set of all functions $f\in\mathbf{F}$ satisfying the equality $f(0)=0.$ It follows directly from the definition that $\mathbf{Am}\subsetneq \mathbf{F}_{0}\subsetneq\mathbf{F}.$

Moreover, a function $f\in\mathbf{F}$ is \emph{increasing} iff the implication $$(x\leqslant y)\Rightarrow (f(x)\leqslant f(y))$$ is valid for all $x,y\in [0, \infty).$

The following theorem was proved in \cite{PTAbAppAn2014}.

\begin{theorem}\label{t2.4}
A function $f \in\mathbf{F}$ is ultrametric preserving if and only if \(f\) is increasing and amenable.
\end{theorem}

\begin{remark}\label{rem_th}
Theorem~\ref{t2.4} was generalized in \cite{Dov2019a} to the special case of the so-called ultrametric distances. These distances were introduced by S.~Priess-Crampe and P.~Ribenboim in 1993 \cite{PR1993AMSUH} and studied in \cite{PR1996AMSUH, PR1997AMSUH, Rib1996PMH, Rib2009JoA}.
\end{remark}

Recall that a function $f \in\mathbf{F}$ is said to be \emph{subadditive} if $$f(x+y)\leqslant f(x)+f(y)$$ holds for all $x,y\in [0, \infty).$ Let us denote by $\mathbf{SI}$ the set of all subadditive increasing functions $f\in\mathbf{Am}.$




Corollary~36 of \cite{Dov24} implies the following result.

\begin{proposition}\label{prop2.1}
The equality
\begin{equation*}
\mathbf{SI}=\mathbf{P}_{\mathbf{U}}\cap\mathbf{P}_{\mathbf{M}}
\end{equation*}
holds.
\end{proposition}


\begin{remark}\label{rem_1.8}
The metric preserving functions can be considered as a special case of metric products (= metric preserving functions of several variables). See, for example, \cite{BD1981, BFS2003BazAaG, DPK2014MS, FS2002, HMCM1991, Kaz2021CoPS}. An important special class of ultrametric preserving functions of two variables was first considered in 2009~\cite{DM2009}.
\end{remark}

\section{Preliminaries on semigroups. Solutions to $\mathbf{F}_{\mathbf{X}}=\mathbf{A}$ for $\mathbf{A}=\mathbf{F},$ $\mathbf{F}_{0}, \mathbf{Am}$}

Let us recall some basic concepts of semigroup theory, see,~for~exam\-ple, ``Fundamentals of Semigroup Theory'' by John M. Howie \cite{Howie1995}.

A \emph{semigroup} is a pair $(S, \ast)$ consisting of a nonempty set $S$ and an associative operation $\ast: S\times S\to S$ which is called the \emph{multiplication} on $S$.
A semigroup $S=(S, \ast)$ is a \emph{monoid} if there is $e\in S$ such that
$$e\ast s=s\ast e=s$$ for every $s\in S.$ 

\begin{definition}\label{def2}
Let $(S, \ast)$ be a semigroup and $\varnothing\ne T\subseteq S.$ Then $T$ is a \emph{subsemigroup}
of $S$ if $a, b\in T$ $\Rightarrow$ $a\ast b\in T.$ If $(S, \ast)$ is a monoid with the identity $e$, then $T$ is a \emph{submonoid} of $S$ if $T$ is a subsemigroup of $S$ and $e\in T.$
\end{definition}

\begin{example}\label{ex1}
The semigroups $(\mathbf{F}, \circ),$ $(\mathbf{Am}, \circ),$ $(\mathbf{P}_{\mathbf{M}}, \circ)$ and $(\mathbf{P}_{\mathbf{U}}, \circ)$ are monoids and the identical mapping $\textrm{id}:[0,\infty)\to[0,\infty),$ $\textrm{id}(x)=x$ for every $x\in[0,\infty),$ is the identity of these monoids.
\end{example}

The following simple lemmas are well known.

\begin{lemma}\label{newlem1}
Let $T$ be a submonoid of a monoid $(S, \ast)$ and let $V\subseteq T.$ Then $V$ is a submonoid of $(S, \ast)$ if and only if $V$ is a submonoid of $T.$
\end{lemma}

\begin{lemma}\label{newlem2}
Let $T_1$ and $T_2$ be submonoids of a monoid $(S, \ast).$ Then the intersection $T_{1}\cap T_{2}$ also is a submonoid of $(S, \ast).$
\end{lemma}

The next theorem describes all solutions to the equation $\mathbf{P}_{\mathbf{X}}=\mathbf{F}.$

\begin{theorem}\label{l1}
The following statements are equivalent for every $\mathbf{X}~\subseteq~\mathbf{M}.$
\begin{itemize}
\item[$(i)$] $\mathbf{X}$ is the empty subclass of $\mathbf{M}.$
\item[$(ii)$] The equality
\begin{equation}\label{eq1_l1}\mathbf{P}_{\mathbf{X}}=\mathbf{F}\end{equation}
holds.
\end{itemize}
\end{theorem}
\begin{proof}
$(i)\Rightarrow (ii).$ Let $\mathbf{X}$ be the empty subclass of $\mathbf{M}.$
Definition~\ref{def1} implies the inclusion $\mathbf{F}\supseteq \mathbf{P}_{\mathbf{X}}.$ Let us consider an arbitrary $f\in\mathbf{F}.$ To prove equality \eqref{eq1_l1} it suffices to show that $f\in\mathbf{P}_{\mathbf{X}}.$ Let us do it. Since $\mathbf{X}$ is empty, the membership relation $(X,d)\in\mathbf{X}$ is false for every metric space $(X, d).$ Consequently, the implication $$((X, d)\in\mathbf{X})\Rightarrow ((X, f\circ d)\in \mathbf{X})$$ is valid for every $(X, d)\in\mathbf{M}.$ It implies $f\in\mathbf{P}_{\mathbf{X}}$ by Definition~\ref{def1}. Equality \eqref{eq1_l1} follows.

$(ii)\Rightarrow (i).$ Let $(ii)$ hold. We must show that $\mathbf{X}$ is empty. Suppose contrary that there is a metric space $(X, d)\in\mathbf{X}.$ Since, by definition, we have $X\ne\varnothing,$ there is a point $x_{0}\in X.$ Consequently, $d(x_0, x_0)=0$ holds. Let $c\in (0, \infty)$ and let $f: [0, \infty)\to [0, \infty)$ be a constant function,
\begin{equation*}\label{eq2_l1}
f(t)=c
\end{equation*}
for every $t\in[0, \infty).$ In particular, we have
\begin{equation}\label{eq3_l1}
f(0)=c>0.
\end{equation}
Equality \eqref{eq1_l1} implies that $f\circ d$ is a metric on $X.$ Thus, we have $$0=f(d(x_0, x_0))=f(0),$$ which contradicts \eqref{eq3_l1}. Statement $(i)$ follows.
\end{proof}

\begin{remark}
Theorem~\ref{l1} becomes invalid if we allow the empty metric space to be considered. The equality $$\mathbf{P}_{\mathbf{X}}=\mathbf{F}$$ holds if the nonempty class $\mathbf{X}$ contains only the empty metric space.
\end{remark}

Let us describe now all possible solutions to $\mathbf{P}_{\mathbf{X}}=\mathbf{F}_{0}.$

\begin{theorem}\label{Th(1)}
The equality
\begin{equation}\label{th_eq1}
\mathbf{P}_{\mathbf{X}}=\mathbf{F}_{0}
\end{equation}
holds if and only if $\mathbf{X}$ is a nonempty subclass of $\mathbf{M}_{1}.$
\end{theorem}

\begin{proof}
Let $\mathbf{X}\subseteq\mathbf{M}_{1}$ be nonempty. Equality \eqref{th_eq1} holds iff
\begin{equation}\label{th_eq1_1}
\mathbf{P}_{\mathbf{X}}\supseteq\mathbf{F}_{0}
\end{equation}
and
\begin{equation}\label{th_eq1_2}
\mathbf{P}_{\mathbf{X}}\subseteq\mathbf{F}_{0}.
\end{equation}
Let us prove the validity of \eqref{th_eq1_1}. Let $f\in\mathbf{F}_{0}$ be arbitrary. Since every $(X, d)\in\mathbf{X}$ is an one-point metric space, we have $f\circ d=d$ for all $(X, d)\in\mathbf{X}$ by positivity property of metric spaces, Inclusion \eqref{th_eq1_1} follows.

Let us prove \eqref{th_eq1_2}. The inclusion $\mathbf{P}_{\mathbf{X}}\subseteq\mathbf{F}$ follows from Definition~\ref{def1}. Thus, if \eqref{th_eq1_2} does not hold, then there is $f_{0}\in\mathbf{F}$ such that $f_{0}\in\mathbf{P}_{\mathbf{X}},$
\begin{equation}\label{lem_eq3}
f_{0}(0)=k\quad\mbox{and}\quad k>0.
\end{equation}
Since $\mathbf{X}$ is nonempty, there is $(X_{0}, d_{0})\in\mathbf{X}.$ Let $x_0$ be a (unique) point of $X_{0}.$ Since $f_0$ belongs to $\mathbf{P}_{\mathbf{X}},$ the function $f_{0}\circ d_{0}$ is a metric on $X_0.$ Now, using \eqref{lem_eq3}, we obtain
\begin{equation*}
f_{0}(d_{0}(x_{0}, x_{0}))=f_{0}(0)=k>0,
\end{equation*}
which contradicts the positivity property of metric spaces. Inclusion \eqref{th_eq1_2} follows.

Let \eqref{th_eq1} hold. We must show that $\mathbf{X}$ is a nonempty subclass of $\mathbf{M}_{1}.$ If $\mathbf{X}$ is empty, then
\begin{equation}\label{th_eq1_3}
\mathbf{P}_{\mathbf{X}}=\mathbf{F}
\end{equation}
holds by Theorem~\ref{l1}. Equality \eqref{th_eq1_3} contradicts equality \eqref{th_eq1}. Hence, $\mathbf{X}$ is nonempty. To complete the proof we must show that
\begin{equation}\label{th_eq1_4}
\mathbf{X}\subseteq\mathbf{M}_{1}.
\end{equation}
Let us consider the constant function $f_{0}:[0, \infty)\to[0, \infty)$ such that
\begin{equation}\label{th_eq1_5}
f_{0}(t)=0
\end{equation}
for every $t\in[0,\infty).$ Then $f_{0}$ belongs to $\mathbf{F}_{0}.$ Hence, for every $(X, d)\in\mathbf{X}$, the mapping $d_{0}:=f_{0}\circ d$ is a metric on $X.$ Now \eqref{th_eq1_5} implies $d_{0}(x,y)=0$ for all $x,y\in X$ and $(X, d)\in\mathbf{X}.$ Hence, $\card(X)=1$ holds, because the metric space $(X, d_{0})$ is one-point by positivity property. Inclusion \eqref{th_eq1_4} follows. The proof is completed.
\end{proof}

The next theorem gives us all solutions to the equation $\mathbf{P}_{\mathbf{X}}=\mathbf{Am}.$

\begin{theorem}\label{Th(2)}
The following statements are equivalent for every $\mathbf{X}~\subseteq~\mathbf{M}.$
\begin{itemize}
\item[$(i)$] The inclusion \begin{equation}\label{nth_1}\mathbf{X}\subseteq\mathbf{Dis}\end{equation}
holds, and there is $(Y, \rho)\in\mathbf{X}$ with \begin{equation}\label{nth_2}\card(Y)\geqslant 2,\end{equation} and we have
 \begin{equation}\label{nth_3}\mathbf{Dis}_{X_{1}}\subseteq\mathbf{X}\end{equation} for every $(X_1, d_1)\in\mathbf{X}.$
\item[$(ii)$] The equality \begin{equation}\label{nth_4}\mathbf{P}_{\mathbf{X}}=\mathbf{Am}\end{equation} holds.
\end{itemize}
\end{theorem}

\begin{proof}
$(i)\Rightarrow (ii).$ Let $(i)$ hold. Equality \eqref{nth_4} holds iff
\begin{equation}\label{nth_5}
\mathbf{P}_{\mathbf{X}}\supseteq\mathbf{Dis}
\end{equation}
and
\begin{equation}\label{nth_6}
\mathbf{P}_{\mathbf{X}}\subseteq\mathbf{Dis}.
\end{equation}
Let us prove \eqref{nth_5}. Inclusion \eqref{nth_5} holds iff we have
\begin{equation}\label{nth_7}
(X_{1}, f\circ d_{1})\in\mathbf{X}
\end{equation}
for all $f\in\mathbf{Am}$ and $(X_1, d_1)\in\mathbf{X}.$ Relation \eqref{nth_7} follows from Theorem~\ref{Th(1)} if $(X_1, d_1)\in\mathbf{M}_{1}.$ To see it we only note that $\mathbf{Am}\subseteq\mathbf{F}_{0}.$ Let us consider the case when $$\card(X_1)\geqslant 2.$$ Since $(X_1, d_1)$ is discrete by \eqref{nth_1}, Definition~\ref{defdiscr} implies that there is $k_{1}\in (0, \infty)$ satisfying $$d_{1}(x,y)=k_{1}$$ for all distinct $x, y\in X_{1}.$ Let $f\in\mathbf{Am}$ be arbitrary. Then $f(k_1)$ is strictly positive and $$f(d_{1}(x,y))=f(k_1)$$ holds for all distinct $x,y\in X_{1}.$ Thus, $f\circ d_{1}$ is discrete metric on $X_{1},$ i.e. we have
\begin{equation}\label{nth_8}
(X_{1}, f\circ d_{1})\in\mathbf{Dis}_{X_1}.
\end{equation}
Now, \eqref{nth_7} follows from \eqref{nth_3} and \eqref{nth_8}.

Let us prove \eqref{nth_6}. To do it we must show that every $f\in\mathbf{P}_{\mathbf{X}}$ is amenable.

Suppose contrary that $f$ belongs to $\mathbf{P}_{\mathbf{X}}$ but the equality
\begin{equation}\label{nth_9}
f(t_1)=0
\end{equation}
holds with some $t_{1}\in (0, \infty).$ By statement $(i)$ we can find $(Y, \rho)\in\mathbf{X}$ such that \eqref{nth_2} and $$\rho(x,y)=t_{1}$$ hold for all distinct $x,y\in Y.$ Now $f\in\mathbf{P}_{\mathbf{X}}$ and $(Y, \rho)\in\mathbf{X}$ imply that $f\circ\rho$ is a metric on $Y.$ Consequently, for all distinct $x,y\in Y,$ we have $$f(\rho(x,y))=f(t_1)>0,$$ which contradicts \eqref{nth_9}. The validity of \eqref{nth_6} follows.

$(ii)\Rightarrow (i).$ Let $\mathbf{X}$ satisfy equality \eqref{eq3_pr}. Since $\mathbf{Am}\ne\mathbf{F}$ holds, the class $\mathbf{X}$ is nonempty by Theorem~\ref{l1}. Moreover, using Theorem~\ref{Th(1)} we see that $\mathbf{X}$ contains a metric space $(X, d)$ with $\card(X)\geqslant 2,$ because $\mathbf{Am}\ne\mathbf{F}_{0}.$

If the inequality $$\card(Y)\leqslant 2$$ holds for every $(Y, \rho)\in\mathbf{X},$ then all metric spaces belonging to $\mathbf{X}$ are discrete (see Example~\ref{ex(9)}). Using the definitions of $\mathbf{Dis}$ and $\mathbf{Am},$ it is easy to prove that for each $(X_{1}, d_{1})\in\mathbf{Dis}$ and every $(X_{1}, d)\in\mathbf{Dis}_{X_1}$ there exists $f\in\mathbf{Am}$ such that $d=f\circ d_{1}.$ Hence to complete the proof it suffices to show that every $(X, d)\in\mathbf{X}$ is discrete when
\begin{equation}\label{nth_10}
\card(X)\geqslant 3.
\end{equation}

Let us consider arbitrary $(X,d)\in\mathbf{X}$ satisfying \eqref{nth_10}. Suppose that $(X,d)\not\in \mathbf{Dis}$. Then by Proposition~\ref{Pr(10)} there are distinct $a,b,c,\in X$ such that
\begin{equation}\label{nth_11}
d(a,b)\ne d(b,c)\ne d(c,a).
\end{equation}
Let $c_{1}$ and $c_{2}$ be points of $(0, \infty)$ such that
\begin{equation}\label{nth_12}
c_{2}>2c_{1}.
\end{equation}
Now we can define $f\in\mathbf{Am}$ as
\begin{equation}\label{nth_13}
f(t) := \left\{
\begin{array}{ll}
0 & \quad \hbox{if}\quad t = 0, \\
c_{2} & \quad \hbox{if}\quad t = d(b,c), \\
c_{1} & \quad \hbox{otherwise}.
\end{array}
\right.
\end{equation}
Equality \eqref{nth_4} implies that $f\circ d$ is a metric on $X$. Consequently, we have
\begin{equation}\label{nth_14}
f(d(b,c))\leqslant f(d(b,a))+f(d(b,c))
\end{equation}
by triangle inequality. Now using \eqref{nth_11} and \eqref{nth_13} we can rewrite \eqref{nth_14} as $$c_{2}\leqslant c_{1}+c_{1},$$
which contradicts \eqref{nth_2}. It implies $(X,d)\in\mathbf{Dis}.$ The proof is completed.
\end{proof}

\begin{corollary} The equalities
$$\mathbf{P}_{\mathbf{Dis}}=\mathbf{P}_{\mathbf{M}_{2}}=\mathbf{Am}$$
hold.
\end{corollary}

\begin{remark}
The equality
\begin{equation*}
\mathbf{P}_{\mathbf{M}_{2}}=\mathbf{Am}
\end{equation*}
is known, see, for example, Remark~1.2 in paper \cite{DM2013}. This paper contains also ``constructive'' characterizations of the smallest bilateral ideal and the largest subgroup of the monoid $\mathbf{P}_{\mathbf{M}}.$
\end{remark}

\begin{proposition}\label{propvsp}
Let $\mathbf{X}$ be a subclass of $\mathbf{M}.$ Then $\mathbf{P}_{\mathbf{X}}$ is a submonoid of $(\mathbf{F}, \circ).$
\end{proposition}

\begin{proof}
It follows directly from Definition~\ref{def1} that
$$\mathbf{P}_{\mathbf{X}}\subseteq\mathbf{F}$$
holds and that the identity mapping $\textrm{id}:[0,\infty)\to[0,\infty)$ belongs to $\mathbf{P}_{\mathbf{X}}.$ Hence, by Lemma~\ref{newlem1}, it is suffices to prove
\begin{equation}\label{eqv1}
f\circ g\in \mathbf{P}_{\mathbf{X}}
\end{equation}
for all $f, g\in\mathbf{P}_{\mathbf{X}}.$

Let us consider arbitrary $f\in\mathbf{P}_{\mathbf{X}}$ and $g\in\mathbf{P}_{\mathbf{X}}.$ Then, using Definition~\ref{def1}, we see that $(X, g\circ d)$ belongs to $\mathbf{X}$ for every $(X, d)\in\mathbf{X}.$ Consequently,
\begin{equation}\label{eqv2}
(X, f\circ (g\circ d))\in\mathbf{X}
\end{equation}
holds. 
Since the composition of functions is always associative, the equality
\begin{equation}\label{eqv3}
(f\circ g)\circ d=f\circ (g\circ d)
\end{equation}
holds for every $(X, d)\in\mathbf{X}.$ Now \eqref{eqv1} follows from \eqref{eqv2} and \eqref{eqv3}.
\end{proof}

The above proposition implies the following corollary.

\begin{corollary}\label{cor_4}
If the equation $$\mathbf{P}_{\mathbf{X}}=\mathbf{A}$$ has a solution, then $\mathbf{A}$ is a submonoid of $\mathbf{F}.$
\end{corollary}

The following example shows that the converse of Corollary~\ref{cor_4} is, generally speaking, false.

\begin{example}\label{ex(10)}
Let us define $\mathbf{A}_{1}\subseteq\mathbf{F}$ as
\begin{equation*}
\mathbf{A}_{1}=\{f_{1}, \id\},
\end{equation*}
where $f_{1}\in\mathbf{F}$ is defined such that
\begin{equation}\label{eq_f}
f_{1}(t) := \left\{
\begin{array}{ll}
1 & \quad \hbox{if}\quad t = 0, \\
0 & \quad \hbox{if}\quad t = 1, \\
t & \quad \hbox{otherwise}
\end{array}
\right.
\end{equation}
and $\id$ is the identical mapping of $[0, \infty).$ The equalities $f_{1}\circ f_{1}=\id,$ $f_{1}\circ\id=f_{1}=\id\circ f_{1}$ show that $\mathbf{A}_{1}$ is a submonoid of $(\mathbf{F}, \circ)$. Suppose that there is $\mathbf{X}_{1}\subseteq\mathbf{M}$ satisfying the equality
\begin{equation}\label{eq_ex2}
\mathbf{P}_{\mathbf{X}_{1}}=\mathbf{A}_{1}.
\end{equation}
Then using Theorem~\ref{l1}, we see that $\mathbf{X}_{1}$ is nonempty because $\mathbf{A}_{1}\ne\mathbf{F}$ holds. Let $(X_1, d_1)$ be an arbitrary metric space from $\mathbf{A}_{1}.$ Since $X_1$ is nonempty, we can find $x_{1}\in X_{1}$. Then \eqref{eq_ex2} implies that $f_{1}\circ d_{1}$ is metric on $X_1.$ Consequently, we have
\begin{equation*}
f_{1}(d_{1}(x_1, x_1))=f_{1}(0)=0,
\end{equation*}
which contradicts \eqref{eq_f}.
\end{example}

\section{Submonoids of monoids $\mathbf{P}_{\mathbf{M}}$ and $\mathbf{P}_{\mathbf{U}}$}

The following theorem gives a solution to Problem~\ref{probl1}.

\begin{theorem}\label{mainth}
Let $\mathbf{A}$ be a nonempty subset of the set $\mathbf{P}_{\mathbf{M}}$ of all metric preserving functions. Then the following statements are equivalent.
\begin{itemize}
\item[$(i)$] The equality
\begin{equation}\label{theq_1}
\mathbf{P}_{\mathbf{X}}=\mathbf{A}
\end{equation}
has a solution $\mathbf{X}\subseteq \mathbf{M}.$
\item[$(ii)$] $\mathbf{A}$ is a submonoid of $(\mathbf{F}, \circ).$
\item[$(iii)$] $\mathbf{A}$ is a submonoid of $(\mathbf{P}_{\mathbf{M}}, \circ).$
\end{itemize}
\end{theorem}

\begin{proof}
$(i)\Rightarrow (ii).$ Suppose that there is $\mathbf{X}\subseteq \mathbf{M}$ such that \eqref{theq_1} holds. Then $\mathbf{A}$ is a submonoid of $(\mathbf{F}, \circ)$ by Proposition~\ref{propvsp}. 



$(ii)\Rightarrow (iii).$ Let $\mathbf{A}$ be a submonoid of $(\mathbf{F}, \circ).$ By Proposition~\ref{propvsp}, the monoid $(\mathbf{P}_{\mathbf{M}}, \circ)$ also is a submonoid of $(\mathbf{F}, \circ).$ Then using the inclusion $\mathbf{A}\subseteq\mathbf{P}_{\mathbf{M}}$ we obtain that $\mathbf{A}$ is a submonoid of $(\mathbf{P}_{\mathbf{M}}, \circ)$ by Lemma~\ref{newlem1}.

$(iii)\Rightarrow (i).$ Let $\mathbf{A}$ be a submonoid of $(\mathbf{P}_{\mathbf{M}}, \circ).$ We must prove that \eqref{theq_1} has a solution $\mathbf{X}\subseteq\mathbf{M}.$

Let $(X,d)$ be a metric space such that
\begin{equation}\label{def}
\{d(x,y): x,y\in X\} =[0, \infty).
\end{equation}

Write
\begin{equation}\label{vsp7}
\mathbf{X}:=\{(X, f\circ d): f\in \mathbf{A}\}.
\end{equation}
We claim that \eqref{theq_1} holds if $\mathbf{X}$ is defined by \eqref{vsp7}. To prove it we note that \eqref{theq_1} holds iff
\begin{equation}\label{vsp8}
\mathbf{A}\subseteq\mathbf{P}_{\mathbf{X}}
\end{equation}
and
\begin{equation}\label{vsp9}
\mathbf{A}\supseteq\mathbf{P}_{\mathbf{X}}.
\end{equation}

Let us prove \eqref{vsp8}. This inclusion holds if for every $f\in\mathbf{A}$ and each $(Y, \rho)\in \mathbf{X}$ we have $(Y, f\circ \rho)\in \mathbf{X}.$ Let us consider arbitrary $(Y, \rho)\in \mathbf{X}$ and $f\in\mathbf{A}.$ Then, using \eqref{vsp7}, we can find $g\in\mathbf{A}$ such that
\begin{equation}\label{vsp10}
X=Y \quad\mbox{and}\quad \rho=g\circ d.
\end{equation}
Since $\mathbf{A}$ is a monoid, the membership relations $f\in\mathbf{A}$ and $g\in\mathbf{A}$ imply $g\circ f\in\mathbf{A}.$ Hence, we have
\begin{equation}\label{vsp11}
(X, g\circ f\circ d)\in\mathbf{X}
\end{equation}
by \eqref{vsp7}. Now $(Y, f\circ\rho)\in\mathbf{X}$ follows from \eqref{vsp10} and \eqref{vsp11}.

Let us prove \eqref{vsp9}. Let $g_{1}$ belong to $\mathbf{P}_{\mathbf{X}}$ and let $(X,d)$ be the same as in \eqref{vsp7}. Then $(X, g_{1}\circ d)$ belongs to $\mathbf{X}$ and, using \eqref{vsp7}, we can find $f_{1}\in\mathbf{A}$ such that
\begin{equation}
(X, g_{1}\circ d)=(X, f_{1}\circ d).
\end{equation}
The last equality implies
\begin{equation}\label{**}
g_{1}(d(x,y))=f_{1}(d(x,y))
\end{equation}
for all $x,y\in X.$ Consequently, $g_{1}(t)=f_{1}(t)$ holds for every $t\in[0, \infty)$ by \eqref{def}. Thus, we have $g_{1}=f_{1}.$ That implies $g_{1}\in\mathbf{A}.$ Inclusion \eqref{vsp9} follows. The proof is completed.
\end{proof}

Let us turn now to Question~\ref{quest_1}. Proposition~\ref{prop2.1} and Lemma~\ref{newlem2} give us the following result.






\begin{theorem}\label{mainth_2}
There is $\mathbf{X}\subseteq\mathbf{M}$ such that
\begin{equation}\label{si}
\mathbf{P}_{\mathbf{X}}=\mathbf{SI}.
\end{equation}
\end{theorem}

\begin{proof}
By Proposition~\ref{propvsp}, the monoids $(\mathbf{P}_{\mathbf{M}}, \circ)$ and $(\mathbf{P}_{\mathbf{U}}, \circ)$ are submonoids of $(\mathbf{F}, \circ)$. The equality
\begin{equation}\label{si*}
\mathbf{SI}=\mathbf{P}_{\mathbf{M}}\cap\mathbf{P}_{\mathbf{U}}
\end{equation}
holds by Proposition~\ref{prop2.1}. Using \eqref{si*} and Lemma~\ref{newlem2} with $T_{1}=\mathbf{P}_{\mathbf{M}},$ $T_{2}=\mathbf{P}_{\mathbf{U}}$ and $\mathbf{S}=\mathbf{F}$ we see that $\mathbf{SI}$ also is a submonoid of $\mathbf{F}.$ Consequently, Theorem~\ref{mainth} with $\mathbf{A}=\mathbf{SI}$ implies that there is $\mathbf{X}\subseteq \mathbf{M}$ such that \eqref{si} holds.
\end{proof}

The next theorem is an ultrametric analog of Theorem~\ref{mainth} and it gives us a solution to Problem~\ref{probl3}.

\begin{theorem}\label{mainth_4}
Let $\mathbf{A}$ be a nonempty subset of the set $\mathbf{P}_{\mathbf{U}}$ of all ultrametric preserving functions. Then the following statements are equivalent.
\begin{itemize}
\item[$(i)$] The equality
$
\mathbf{P}_{\mathbf{X}}=\mathbf{A}
$
has a solution $\mathbf{X}\subseteq \mathbf{U}.$
\item[$(ii)$] $\mathbf{A}$ is a submonoid of $(\mathbf{F}, \circ).$
\item[$(iii)$] $\mathbf{A}$ is a submonoid of $(\mathbf{P}_{\mathbf{U}}, \circ).$
\end{itemize}
\end{theorem}

A proof of Theorem~\ref{mainth_4} can be obtained by a simple modification of the proof of Theorem~\ref{mainth}. We only note that the ultrametric space defined in Example~\ref{example_19} satisfies the equality \eqref{def} with $X=\mathbb R_{0}^{+}$ and $d=d^{+}$.



\section{Two conjectures}

\begin{conjecture}\label{con1}
The equality
\begin{equation*}
\mathbf{P}_{\mathbf{X}}=\mathbf{A}
\end{equation*}
has a solution $\mathbf{X}\subseteq\mathbf{M}$ for every submonoid $\mathbf{A}$ of the monoid $\mathbf{Am}.$
\end{conjecture}

Example~\ref{ex(10)} shows that we cannot replace $\mathbf{Am}$ with $\mathbf{F}$ in Conjecture~\ref{con1}, but we hope that the following is valid.

\begin{conjecture}\label{con2}
For every submonoid $\mathbf{A}$ of the monoid $\mathbf{F}$ there exists $\mathbf{X}\subseteq\mathbf{M}$ such that $\mathbf{P}_{\mathbf{X}}$ and $\mathbf{A}$ are isomorphic submonoids.
\end{conjecture}



\section*{Funding information}

The first author was partially supported by a grant from the Simons Foundation (Award 1160640, Presidential Discretionary-Ukraine Support Grants, Viktoriia Bilet). Oleksiy Dovgoshey was supported by a grant of Turku University, Finland.


\begin{thebibliography}{10}
\expandafter\ifx\csname url\endcsname\relax
  \def\url#1{\texttt{#1}}\fi
\expandafter\ifx\csname urlprefix\endcsname\relax\def\urlprefix{URL }\fi
\expandafter\ifx\csname href\endcsname\relax
  \def\href#1#2{#2} \def\path#1{#1}\fi

\bibitem{Dov24}
O.~Dovgoshey, Strongly ultrametric preserving functions, arXiv:2401.15922v2
  (2024) 1--25.

\bibitem{JTRRACEFNSAMR2020}
J.~Jachymski, F.~Turobo\'{s}, On functions preserving regular semimetrics and
  quasimetrics satisfying the relaxed polygonal inequality, Rev. R. Acad.
  Cienc. Exactas F\'{\i}s. Nat., Ser. A Mat., RACSAM 114~(3) (2020) 159.

\bibitem{DLPS2008TaiA}
C.~Delhomm\'{e}, C.~Laflamme, M.~Pouzet, N.~Sauer, {Indivisible ultrametric
  spaces}, Topology and its Applications 155~(14) (2008) 1462--1478.

\bibitem{Sea2007}
M.~O. Searc\'{o}id, {Metric Spaces}, Springer---Verlag, London, 2007.

\bibitem{Wilson1935}
W.~A. Wilson, {On certain types of continuous transformations of metric
  spaces}, American Journal of Mathematics 57~(1) (1935) 62--68.
\newblock \href {https://doi.org/10.2307/2372019} {\path{doi:10.2307/2372019}}.

\bibitem{Blumenthal1936}
L.~Blumenthal, Remarks concerning the euclidean four-point property, Ergeb.
  Math. Kolloq. Wien 7 (1936) 7--10.

\bibitem{BD1981MS}
J.~Bors\'{\i}k, J.~Dobo\v{s}, Functions the composition with a metric of which
  is a metric, Math. Slovaca 31 (1981) 3--12.

\bibitem{Borsik1988}
J.~Bors\'{\i}k, J.~Dobo\v{s}, {On metric preserving functions}, Real Analysis
  Exchange 13~(1) (1988) 285--293.

\bibitem{Dobos1996}
J.~Dobo\v{s}, {On modification of the Euclidean metric on reals}, Tatra Mt.
  Math. Publ. 8 (1996) 51--54.

\bibitem{Dobos1994}
J.~Dobo\v{s}, Z.~Piotrowski, {Some remarks on metric-preserving functions},
  Real Anal. Exchange 19~(1) (1994) 317--320.

\bibitem{Dobos1996a}
J.~Dobo\v{s}, Z.~Piotrowski, {A note on metric preserving functions}, Int. J.
  Math. Math. Sci. 19~(1) (1996) 199--200.
\newblock \href {https://doi.org/http://dx.doi.org/10.1155/S0161171296000282}
  {\path{doi:http://dx.doi.org/10.1155/S0161171296000282}}.

\bibitem{Dobos1997}
J.~Dobo\v{s}, Z.~Piotrowski, When distance means money, Internat. J. Math. Ed.
  Sci. Tech. 28 (1997) 513--518.

\bibitem{DM2013}
O.~Dovgoshey, O.~Martio, {Functions transferring metrics to metrics},
  Beitr{\"a}ge zur Algebra und Geometrie 54~(1) (2013) 237--261.
\newblock \href {https://doi.org/10.1007/s13366-011-0061-7}
  {\path{doi:10.1007/s13366-011-0061-7}}.

\bibitem{V1997RAE}
R.~W. Vallin, Metric preserving functions and differentiation, Real Anal. Exch.
  22~(1) (1997) 86.

\bibitem{V1998AMUC}
R.~W. Vallin, On metric preserving functions and infinite derivatives, Acta
  Math. Univ. Comen., New Ser. 67~(2) (1998) 373--376.

\bibitem{V1998IJMMS}
R.~W. Vallin, A subset of metric preserving functions, Int. J. Math. Math. Sci.
  21~(2) (1998) 409--410.

\bibitem{Vallin2000}
R.~W. Vallin, {Continuity and differentiability aspects of metric preserving
  functions}, Real Anal. Exchange 25~(2) (1999/2000) 849--868.

\bibitem{PT2014FPTA}
P.~Pongsriiam, I.~Termwuttipong, On metric-preserving functions and fixed point
  theorems, Fixed Point Theory Appl. 2014 (2014) 14, id/No 179.
\newblock \href {https://doi.org/10.1186/1687-1812-2014-179}
  {\path{doi:10.1186/1687-1812-2014-179}}.

\bibitem{V2002TMMP}
R.~W. Vallin, On preserving {{\((\mathbb{R},\text{Eucl. })\)}} and almost
  periodic functions, Tatra Mt. Math. Publ. 24~(1) (2002) 1--6.

\bibitem{Dobos1998}
J.~Dobo\v{s}, {Metric Preserving Functions}, \v{S}troffek, Ko\v{s}ice,
  Slovakia, 1998.

\bibitem{Corazza1999}
P.~Corazza, {Introduction to metric-preserving functions}, Amer. Math. Monthly
  104~(4) (1999) 309--323.

\bibitem{PTAbAppAn2014}
P.~Pongsriiam, I.~Termwuttipong,
  \href{https://projecteuclid.org/euclid.aaa/1412276956}{{Remarks on
  ultrametrics and metric-preserving functions}}, Abstr. Appl. Anal. 2014
  (2014) 1--9.
\newblock \href {https://doi.org/10.1155/2014/163258}
  {\path{doi:10.1155/2014/163258}}.
\newline\urlprefix\url{https://projecteuclid.org/euclid.aaa/1412276956}

\bibitem{Dov2020MS}
O.~Dovgoshey, On ultrametric-preserving functions, Math. Slovaca 70~(1) (2020)
  173--182.

\bibitem{VD2021MS}
R.~W. Vallin, O.~A. Dovgoshey, P-adic metric preserving functions and their
  analogues, Math. Slovaca 71~(2) (2021) 391--408.

\bibitem{Dov2019a}
O.~Dovgoshey, Combinatorial properties of ultrametrics and generalized
  ultrametrics, Bull. Belg. Math. Soc. Simon Stevin 27~(3) (2020) 379--417.

\bibitem{PR1993AMSUH}
S.~Priess-Crampe, P.~Ribenboim, {Fixed points, combs and generalized power
  series}, Abh. Math. Sem. Univ. Hamburg 63 (1993) 227--244.

\bibitem{PR1996AMSUH}
S.~Priess-Crampe, P.~Ribenboim, {Generalized ultrametric spaces I.}, Abh. Math.
  Sem. Univ. Hamburg 66 (1996) 55--73.
\newblock \href {https://doi.org/https://doi.org/10.1007/BF02940794}
  {\path{doi:https://doi.org/10.1007/BF02940794}}.

\bibitem{PR1997AMSUH}
S.~Priess-Crampe, P.~Ribenboim, {Generalized ultrametric spaces II}, Abh. Math.
  Sem. Univ. Hamburg 67 (1997) 19--31.

\bibitem{Rib1996PMH}
P.~Ribenboim, {The new theory of ultrametric spaces}, Periodica Math. Hung.
  32~(1--2) (1996) 103--111.

\bibitem{Rib2009JoA}
P.~Ribenboim, {The immersion of ultrametric spaces into Hahn Spaces}, J. of
  Algebra 323~(5) (2009) 1482--1493.

\bibitem{BD1981}
J.~Bors\'{\i}k, J.~Dobo\v{s}, On a product of metric spaces, Math. Slovaca 31
  (1981) 193--205.

\bibitem{BFS2003BazAaG}
A.~Bernig, T.~Foertsch, V.~Schroeder, Non standard metric products,
  Beitr\"{a}ge zur Algebra und Geometrie 44~(2) (2003) 499--510.

\bibitem{DPK2014MS}
O.~Dovgoshey, E.~Petrov, G.~Kozub, Metric products and continuation of isotone
  functions, Math. Slovaca 64~(1) (2014) 187--208.

\bibitem{FS2002}
T.~Foertsch, V.~Schroeder, Minkowski-versus {E}uclidean rank for products of
  metric spaces, Adv. Geom. 2 (2002) 123--131.

\bibitem{HMCM1991}
I.~Herburt, M.~Moszy\'{n}ska, {On metric products}, Colloq. Math. 62 (1991)
  121--133.

\bibitem{Kaz2021CoPS}
D.~Kazukawa, Construction of product spaces, Anal. Geom. Metr. Spaces 9~(1)
  (2021) 186--218.

\bibitem{DM2009}
O.~Dovgoshey, O.~Martio, {Products of metric spaces, covering numbers, packing
  numbers and characterizations of ultrametric spaces}, Rev. Roumaine Math.
  Pures. Appl. 54~(5-6) (2009) 423--439.

\bibitem{Howie1995}
J.~M. Howie, Fundamentals of semigroup theory, Clarendan Press, Oxford, 1995.

\end{thebibliography}
\end{document}